\documentclass[11pt]{amsart}

\usepackage{color}

\usepackage{ amsmath,  amssymb,  amsthm,  amsfonts,  amsbsy,  amstext}
\usepackage{amsmath,   amssymb} 
\usepackage{amsthm,   amsfonts,   mathrsfs}
\usepackage{amsfonts,  graphicx}

\theoremstyle{plain}
\newtheorem{Theo}{Theorem}[section]
\newtheorem{lemm}[Theo]{Lemma}

\newtheorem{prop}[Theo]{Proposition}

\theoremstyle{plain}
\theoremstyle{definition}

\theoremstyle{remark}
\newtheorem{Rema}[Theo]{Remark}
\newtheorem*{rema*}{Remark}

\newcommand{\ZZ}{\mathbb{Z}}  
  
\newcommand{\NN}{\mathbb{N}}

\newcommand{\RR}{\mathbb{R}} 
 
\newcommand{\DD}{\textnormal{D}}

\newcommand{\ap}{\text{\it a priori}}
\numberwithin{equation}{section}

\date{}

\begin{document}

\title[Global well-posedness for  Boussinesq System]
{Global well-posedness for a   Boussinesq- Navier-Stokes  System with critical dissipation}
\author[T. Hmidi] {Taoufik  Hmidi}
\address{IRMAR,   Universit\'e de Rennes 1\\ Campus de
Beaulieu\\ 35~042 Rennes cedex\\ France}
\email{taoufik.hmidi@univ-rennes1.fr}
\author[S. Keraani]{Sahbi  Keraani}
\address{IRMAR,   Universit\'e de Rennes 1\\ Campus de
Beaulieu\\ 35~042 Rennes cedex\\ France}
\email{sahbi.keraani@univ-rennes1.fr}
\author[F. Rousset]{Fr\'ed\'eric Rousset}
\address{IRMAR,   Universit\'e de Rennes 1\\ Campus de
Beaulieu\\ 35~042 Rennes cedex\\ France}
\email{frederic.rousset@univ-rennes1.fr}
\subjclass[2000]{ }

\begin{abstract}
In this paper we study a fractional diffusion Boussinesq model which couples
 a Navier-Stokes type equation with fractional diffusion for the velocity and a transport equation for the temperature. We establish  global well-posedness results with rough initial data.
\end{abstract}

\maketitle
\tableofcontents

\section{Introduction} 

The  aim  of  this paper is to study the global well-posedness for  the Boussinesq system with partial 
 fractional dissipation   \begin{equation}
      \label{ns} 
\left\{ \begin{array}{ll} 
\partial_{t}v+v\cdot\nabla v+|\DD|v+\nabla p=\theta e_{2}\\ 
\partial_{t}\theta+v\cdot\nabla\theta=0\\
\textnormal{div}\,  v=0\\
v_{| t=0}=v^{0},   \quad \theta_{| t=0}=\theta^{0}. 
\end{array} \right.
\end{equation}
Here,   we focus on the two-dimensional case,   the space variable $x= (x_{1},   x_{2})$
 is in $\mathbb{R}^2$,   the velocity field $v$    is given by  $v=(v^1,  v^2)$ and the
  pressure $p$ and the temperature $\theta$ are scalar functions.
   The factor $\theta e_{2}$ in the velocity equation,    the vector $e_2$ being  given by $(0,  1)$,   models
for example   the effect of gravity on  the fluid motion.
The operator $|\DD|$ stands for the multiplication by $|\xi| : = \sqrt{\xi_{1}^2+ \xi_{2}^2 }$ in the Fourier space.
    
If we take $\theta^0=0$ then the system \eqref{ns} is reduced to the generalized Navier-Stokes system which was studied in a series of   papers \cite{Wu1,  Wu2,  Wu3} for all space dimension $d\geq2.$ In particular,  
 for the generalized Navier-Stokes system in  dimension two,   
 from the Beale-Kato-Majda criterion  \cite{bkm}  and the  maximum principle for the vorticity \cite{cc},   
 smooth solutions are global in time.
 
 The system \eqref{ns} can be seen as part of the class of generalized Boussinesq systems.
These  systems under the form 
\begin{equation}
\label{bintro}
\left\{ 
\begin{array}{ll} 
\partial_{t}v+v\cdot\nabla v+\nabla p=\theta e_{2} +  \mathcal{D}_{v} v \\ 
\partial_{t}\theta+v\cdot\nabla\theta=  \mathcal{D}_{\theta} \theta\\
\textnormal{div}\,  v=0\\
v_{| t=0}=v^{0},   \quad \theta_{| t=0}=\theta^{0}  
\end{array} \right.
\end{equation}
are simple models widely used in the modelling of oceanic and atmospheric motions.
They  also appear in many physical problems,   we refer for example to \cite{Brenier}  
 for more details about the modelling issues.

 The operators $\mathcal{D}_{v} $ and 
    $\mathcal{D}_{\theta}$ whose form may vary are used to  take into account the possible
    effects of diffusion and dissipation in the fluid motion.  
    
   Mathematically,   the simplest model to study is the fully viscous model  when 
    $\mathcal{D}_{v}= \Delta$,    $\mathcal{D}_{\theta}= \Delta.$ The properties of the system
     are very similar to the one of the two-dimensional Navier-Stokes equation
      and  similar global well-posedness results can be obtained.

  The most difficult  model   for the mathematical study is the inviscid  one,   i.e. when $\mathcal{D}_{v}= \mathcal{D}_{\theta} = 0$.
   A local existence result of smooth solution can be proven  as for  symmetric hyperbolic quasilinear systems,  
    nevertheless,    it is not known if smooth solutions can develop singularities in finite time.
  Indeed,   the temperature $\theta$ is the solution of a transport equation  and the vorticity $\omega=
   \mbox{curl } v = \partial_{1} v^2 - \partial_{2} v^1$ solves the equation
   \begin{equation}
   \label{vortintro}
   \partial_{t} \omega + v \cdot \nabla \omega =  \partial_{1} \theta.
   \end{equation}
   The main difficulty  is that to get an $L^\infty$ estimate on $\omega$ which is crucial to prove
    global existence of smooth solutions for  Euler type  equation,   one needs to 
     estimate $\int_{0}^T ||\partial_{1} \theta ||_{L^\infty}$ and,    unfortunately,   no $\ap$  estimate
      on $\partial_{1} \theta $ is known.

   In order to understand the coupling between  the two equations in Boussinesq type systems,  
      there have been many recent works studying  Boussinesq systems with partial viscosity i.e. with
       a viscous term acting only in one equation. We refer for example to  \cite{ah,  Cha,  dp1,  dp,  dp2,  hk1,  hk}. 
       
     In this paper,   we shall focus on system \eqref{ns}  which corresponds to the case where the  heat conductivity 
      is neglected  and $\mathcal{D}_{v}= - |\DD|$.

   When considering the usual Navier-Stokes equation for the velocity,   i.e. for  $\mathcal{D}_{v}= \Delta$,  
        global well-posedness results were recently established in various functional spaces. In \cite{Cha},   Chae proved the global well-posedness for large  initial data $v^0,  \theta^0\in H^s$ with $s>2.$ This result was improved  by the first two authors \cite{hk1} for less regular initial data,   that is,    $v^0,  \theta^0\in H^s,  $ with $s>0.$ The uniqueness in the  energy space $L^2$  was recently  \mbox{proved in \cite{dp}. }   According to a recent work  of Danchin and Paicu \cite{dp2} one can construct global unique solution when the  dissipation  acts only in  the horizontal direction: $\mathcal{D}_{v}= \partial_{11}.$ 

To explain the new difficulties that appear  when a weaker diffusion  $\mathcal{D}_{v}= - |D|^\alpha,  $  $\alpha <2$
 is considered,   
let us   write the  system under the vorticity formulation.  
 By using  the vorticity  defined as  the scalar $\omega=\partial_1v^2-\partial_2 v^1,  $  we have to study the system 
\begin{equation*}
\left\{ \begin{array}{ll} 
\partial_{t}\omega+v\cdot\nabla \omega+|\DD|^\alpha \omega=\partial_1\theta,   \\ 
\partial_{t}\theta+v\cdot\nabla\theta=0,  \\
\omega_{| t=0}=\textnormal{curl } v^{0},   \quad \theta_{| t=0}=\theta^{0}. 
\end{array} \right. 
\end{equation*}
The standard $L^2$ energy estimate for this system gives
$$ {d \over dt } \|\omega(t) \|_{L^2}^2  +  \|\omega(t) \|_{\dot{H}^{\alpha  \over 2 } }^2
  \leq \|\theta(t) \|_{\dot{H}^{1- {\alpha \over 2 }}}^2,    \quad \|\theta(t)\|_{L^2} = \|\theta_{0}\|_{L^2}.$$
 When $\alpha=2$,   the combination of these two estimates provides the useful information
  that $\omega \in L^\infty_{loc} L^2\cap L^2_{loc}\dot{H}^{1}.$ When 
   $\alpha <2 $ since no $\ap$  estimate on  $ \|\theta \|_{\dot{H}^{1- {\alpha \over 2 }}} $ is known
     some additional  work is needed  in order to estimate $\omega$.  As in \cite{hk1},   
      the idea would be to use maximal regularity estimates for
        the semi-group $e^{-t |\DD|^\alpha }$ in order to compensate the loss of one
         derivative for $\theta$.
         Nevertheless,   some restriction will appear  in order to control the nonlinear term. 
       For the sake of clarity,    we shall focus on  the most difficult  case $\alpha =1$
       where this approach would fail:  we consider  \eqref{ns} 
            for which  the vorticity form of the system is
         \begin{equation*}
\left\{ \begin{array}{ll} 
\partial_{t}\omega+v\cdot\nabla \omega+|\DD|  \omega=\partial_1\theta,   \\ 
\partial_{t}\theta+v\cdot\nabla\theta=0,  \\
\omega_{| t=0}=\textnormal{curl } v^{0},   \quad \theta_{| t=0}=\theta^{0}. 
\end{array} \right. 
\end{equation*}
     This   is  the   critical case  in the sense  
   that the  gain  of  one derivative by the diffusion term   roughly  compensates  exactly   the loss of one derivative
    in $\theta$ in the vorticity equation and has the same order as the convection term.

  The main result of this paper is a  global well-posedness result for the system \eqref{ns} (see section \ref{preliminaries} for the  definitions  and the basic properties of Besov spaces).
   \begin{Theo}
   \label{theo2}  Let 
 $ \theta^0\in L^2\cap  B_{\infty,  1}^{0}$   and  $v^0$ be a divergence-free vector field belonging to $H^1\cap \dot W^{1,  p} $ with $p\in ]2,  +\infty[.$  Then the  system  \eqref{ns} has a unique global solution $(v,  \theta)$ such that
 $$
 v\in L^\infty_{\textnormal{loc}}(\RR_+; H^1\cap \dot W^{1,  p})\cap  L^1_{\textnormal{loc}}(\RR_+;B_{\infty,  1}^1)\quad and\quad \theta\in L^\infty_{\textnormal{loc}}(\RR_+; L^2\cap B_{\infty,  1}^{0}).
 $$
 
 \end{Theo}

 A few remarks are in order.

\begin{Rema}
From the $\ap$  estimates that we shall get in the proof of Theorem \ref{theo2},   it
 is possible to obtain the existence of various types of global  weak solutions,   this   is discussed 
  in section \ref{sectionbouss}.

At first,   it is possible to get the existence without uniqueness  of   Leray type weak solutions of 
 \eqref{ns} in the energy espace
 $$(v,   \theta)
 \in L^\infty_{loc}(\mathbb{R}_{+},   L^2
 ) \cap L^2_{loc}( \mathbb{R}_{+} \dot{H}^{1\over 2 }) \times L^\infty_{loc}(\mathbb{R}_{+},   L^2) $$
 by using the energy estimate for the system \eqref{ns}. We refer to 
  Proposition \ref{es-energ}.

  We also point out that this system has the following scaling: if $(v,  \theta)$ is a solution of  \eqref{ns} and $\lambda>0$  then $(v_\lambda,  \theta_\lambda)$ is also a solution,   where $v_\lambda(t,  x):=v(\lambda t,  \lambda x)$ and $\theta_\lambda(t,  x):=\lambda\theta(\lambda t,  \lambda x).$ As a consequence,   the space of initial data $\dot H^1\times L^2$ which  is invariant under this transformation  is  critical.

From the $\ap$  estimates that we shall establish (see Proposition \ref{prop500}),    we can also get 
 the existence   of   global  weak solutions  (but stronger than the previous ones) 
  by assuming only   that  $v^0\in H^1$ and $\theta^0\in L^2\cap L^r$ with $r>4$ 
   for example i.e. we can get global weak solutions which almost have a critical regularity.
   Nevertheless the uniqueness for  such solutions which do not satisfy the additional
    regularity assumptions stated in Theorem \ref{theo2}  remains  unsolved. 
 \end{Rema}

\begin{Rema} Our proof gives more time integrability  on the velocity $v$. More precisely ,   we have 
 $v\in \widetilde L^\rho_{\textnormal{loc}}(\RR_+;B_{\infty,  1}^1)$
 for all $\rho\in [1,  \frac p2[.$
\end{Rema}

Let us say a few words about the  main difficulties encountered  in the proof of Theorem \ref{theo2}.
Even if we neglect  the nonlinear terms  in the Boussinesq system,   which reduces the  system to the following one 
$$ \partial_{t} \omega + |\DD|\omega = \partial_{1} \theta,   \quad \partial_{t} \theta  = 0,  $$
it is not clear how   to perform  the standard  energy estimate on   $\omega$   
  without using neither the control of higher order derivatives of $\theta$ nor the maximal regularities of the heat kernel $e^{ - t|\DD|}$ which  would not be compatible with  the nonlinear problem. For example   for the  $L^2$ energy estimate,    one gets 
 $$ {1 \over 2} {d \over dt } \|\omega \|_{L^2}^2  + \|\omega \|_{\dot{H}^{1 \over 2 } }^2
  = \int_{\RR^2} \partial_{1} \theta \,   \omega\,   dx,   \quad \|\theta(t)\|_{L^2} = \|\theta_{0}\|_{L^2},  $$
  from which no conclusion can be made. 
  
The main idea in the proof of Theorem \ref{theo2},   that was  also  successfully used in the study of the Euler-Boussinesq system  \cite{HKR1} is to really use the  structural properties of the system
 solved by $(\omega,   \theta)$.
   We note that the symbol of the system which is given by 
  $$ \mathcal{A}(\xi)  =   \left( \begin{array}{cc}   -  | \xi | &  -i \xi_{1}  \\  0 &  0 \end{array} \right) $$
   is diagonalizable for $\xi \neq0$ with  eigenvalues 
      $0$ and $-|\xi|$  which are  real and simple.
      By using the Riesz transform $\mathcal{R}={ \partial_{1} \over |\DD| } $,  
       one gets  that the diagonal form of the system is given by
       $$ \partial_{t} \big( \omega - \mathcal{R} \theta \big) +  |\DD| \big( \omega - \mathcal{R} \theta\big)
    = 0,   \quad  \partial_{t} \theta = 0.$$
    This last form of the system is much more convenient in order to perform
       $\ap$  estimates.
     For example one gets immediately from the continuity on $L^2$ of $\mathcal{R}$ that
     $$\| \omega(t) \|_{L^2} + \|\theta (t) \|_{L^2} \leq C\big( \| \omega_{0} \|_{L^2} + \|\theta_{0} \|_{L^2}\big).$$ 
     
    To prove Theorem \ref{theo2},    we shall use the same idea,   we shall diagonalize the
     linear part of the system  and  then get  $\ap$  estimates from the study of  the new system.
   The main technical difficulty in  this program when one takes the nonlinear terms into account
    is to evaluate in a sufficiently sharp way the commutator
     $[\mathcal{R},   v \cdot \nabla ] $ between
     the  Riesz transform and  the convection operator. Such commutator estimates are  stated and proven 
      in section \ref{sectionRiesz} of the paper.

   The  remaining of the paper is organized as follows.
   
  Section \ref{preliminaries} is devoted to the definition of the needed functional spaces and
    the statement of some of their useful properties. Some technical lemmas are also given. 
    Section \ref{sectionRiesz} is devoted to the study of some  commutator estimates  involving
    the Riesz transform.
   Section \ref{sectionTD} is dedicated  to the study of linear transport-(fractional) diffusion equation. Basically two kind of estimates are given: smoothing effects and logrithmic estimate. 
    In  section \ref{sectionbouss} we discuss a first set of a priori estimates and  the issue of weak solutions.
   Section \ref{proof} is dedicated to the   proof of Theorem \ref{theo2}.
The last section is devoted to the proof of  some technical lemmas.
      \section{Notations and preliminaries}
    \label{preliminaries}
    \label{Notations and preliminaries}
       \subsection{Notations}Throughout this work we will use the following notations.
       
$\bullet$ For any positive  $A$ and $B$  the notation  $A\lesssim B$ means that there exist a positive  harmless constant $C$ such that $A\le CB$. 

$\bullet$ For any tempered distribution $u$  both $ \widehat u$  and $\mathcal F u$ denote the Fourier transform of $u$.

$\bullet$ Pour every $p\in [1,  \infty]$,   $\|\cdot\|_{L^p}$ denotes the norm in the Lebesgue space $L^p$.

$\bullet$ The norm  in the mixed space time Lebesgue space $L^p([0,  T],  L^r(\mathbb R^d)$ is denoted by  $\|\cdot\|_{L^p_TL^r}$ (with the obvious generalization to  $\|\cdot\|_{L^p_T\mathcal X} $ for any normed space $\mathcal X$).

$\bullet$ For any pair of operators $P$ and $Q$ on some Banach space $\mathcal{X}$,   the commutator $[P,  Q]$ is given by $PQ-QP$.

$\bullet$ For $p\in[1,  \infty]$,   we denote by $\dot{W}^{1,  p}$ the space of distributions $u$ such \mbox{that $\nabla u\in L^p.$ }

    \subsection{Functional spaces}
  Let us introduce the so-called   Littlewood-Paley decomposition and the corresponding  cut-off operators. 
There exists two radial positive  functions  $\chi\in \mathcal{D}(\RR^d)$ and  $\varphi\in\mathcal{D}(\RR^d\backslash{\{0\}})$ such that
\begin{itemize}
\item[\textnormal{i)}]
$\displaystyle{\chi(\xi)+\sum_{q\geq0}\varphi(2^{-q}\xi)=1}$;$\quad \displaystyle{\forall\,  \,  q\geq1,  \,   \textnormal{supp }\chi\cap \textnormal{supp }\varphi(2^{-q})=\varnothing}$\item[\textnormal{ii)}]
 $ \textnormal{supp }\varphi(2^{-j}\cdot)\cap
\textnormal{supp }\varphi(2^{-k}\cdot)=\varnothing,  $ if  $|j-k|\geq 2$.
\end{itemize}

For every $v\in{\mathcal S}'(\RR^d)$ we set 
  $$
\Delta_{-1}v=\chi(\hbox{D})v~;\,   \forall
 q\in\NN,  \;\Delta_qv=\varphi(2^{-q}\hbox{D})v\quad\hbox{ and  }\;
 S_q=\sum_{-1\leq p\leq q-1}\Delta_{p}.$$
 The homogeneous operators are defined by
 $$
 \dot{\Delta}_{q}v=\varphi(2^{-q}\hbox{D})v,  \quad \dot S_{q}v=\sum_{j\leq q-1}\dot\Delta_{j}v,  \quad\forall q\in\ZZ.
 $$
From   \cite{b}  we split the product 
  $uv$ into three parts: $$
uv=T_u v+T_v u+R(u,  v),  
$$
with
$$T_u v=\sum_{q}S_{q-1}u\Delta_q v,  \quad  R(u,  v)=\sum_{q}\Delta_qu\tilde\Delta_{q}v  \quad\hbox{and}\quad \tilde\Delta_{q}=\Delta_{q-1}+\Delta_{q}+\Delta_{q+1}.
$$
  
Let us now define inhomogeneous  Besov spaces. For
 $(p,  r)\in[1,  +\infty]^2$ and $s\in\RR$ we define  the inhomogeneous Besov \mbox{space $B_{p,  r}^s$} as 
the set of tempered distributions $u$ such that
$$\|u\|_{B_{p,  r}^s}:=\Big( 2^{qs}
\|\Delta_q u\|_{L^{p}}\Big)_{\ell ^{r}}<+\infty.$$
The homogeneous Besov space $\dot B_{p,  r}^s$ is defined as the set of  $u\in\mathcal{S}'(\RR^d)$ up to polynomials such that
$$
\|u\|_{\dot B_{p,  r}^s}:=\Big( 2^{qs}
\|\dot\Delta_q u\|_{L^{p}}\Big)_{\ell ^{r}(\ZZ)}<+\infty.
$$
Notice that the usual  Sobolev spaces $H^s$ coincide with $B_{2,  2}^s$ for every $s\in\RR$ and that  the homogeneous spaces $\dot{H}^s$ coincide with $\dot{B}_{2,  2}^s.$

We shall also use need  some mixed space-time spaces. Let $T>0$ \mbox{and $\rho\geq1,  $} we denote by $L^\rho_{T}B_{p,  r}^s$ the space of distributions $u$ such that 
$$
\|u\|_{L^\rho_{T}B_{p,  r}^s}:= \Big\|\Big( 2^{qs}
\|\Delta_q u\|_{L^p}\Big)_{\ell ^{r}}\Big\|_{L^\rho_{T}}<+\infty.$$
We say that 
$u$ belongs to the space
 $\widetilde L^\rho_{T}{B_{p,  r}^s}$ if
 $$
 \|u\|_{ \widetilde L^\rho_{T}{B_{p,  r}^s}}:= \Big( 2^{qs}
\|\Delta_q u\|_{L^\rho_{T}L^p}\Big)_{\ell ^{r}}<+\infty .$$
By  a direct application  of 
the Minkowski inequality,   we have the following links between these spaces. 

Let $ \varepsilon>0,  $ then 
$$
L^\rho_{T}B_{p,  r}^s\hookrightarrow\widetilde L^\rho_{T}{B_{p,  r}^s}\hookrightarrow{L^\rho_{T}}{B_{p,  r}^{s-\varepsilon}}
,  \,  \textnormal{if}\quad  r\geq \rho,  $$
$$
{L^\rho_{T}}{B_{p,  r}^{s+\varepsilon}}\hookrightarrow\widetilde L^\rho_{T}{B_{p,  r}^s}\hookrightarrow L^\rho_{T}B_{p,  r}^s,  \,   \textnormal{if}\quad 
\rho\geq r.
$$
 We will  make continuous use of Bernstein inequalities (see  \cite{che1} for instance).
\begin{lemm}\label{lb}\;
 There exists a constant $C$ such that for $q,  k\in\NN,  $ $1\leq a\leq b$ and for  $f\in L^a(\RR^d)$,   
\begin{eqnarray*}
\sup_{|\alpha|=k}\|\partial ^{\alpha}S_{q}f\|_{L^b}&\leq& C^k\,  2^{q(k+d(\frac{1}{a}-\frac{1}{b}))}\|S_{q}f\|_{L^a},  \\
\ C^{-k}2^
{qk}\|{\Delta}_{q}f\|_{L^a}&\leq&\sup_{|\alpha|=k}\|\partial ^{\alpha}{\Delta}_{q}f\|_{L^a}\leq C^k2^{qk}\|{\Delta}_{q}f\|_{L^a}.
\end{eqnarray*}

\end{lemm}
The following result generalizes the classical Gronwall inequality,   see Lemma 5.2.1 \cite{che1} for the proof. It will be very useful in the proof of the uniqueness part
of \mbox{Theorem \ref{theo2}.}
\begin{lemm}[Osgood lemma]
\label{osgood}
Let  $\gamma\in L^{1}_{\rm loc}(\RR_+;\RR_+)$,    $\mu$ a continuous non decreasing function,   $a\in\RR_+$ and
 $\alpha$  a measurable  function  satisfying  
$$
0\le \alpha(t)\leq a+\int_0^t\gamma(\tau)\mu(\alpha(\tau))d\tau,   \qquad\forall t\in\RR_+.
$$
If we assume that $a>0$ then   
$$
-\mathcal{M}(\alpha(t))+\mathcal{M}(a)\leq\int_0^t\gamma(\tau)d\tau \quad\hbox{with}\quad \mathcal{M}(x):=\int_x^1\frac{dr}{\mu(r)}\cdot
$$
If we assume  $a=0$  and $\lim_{x\to 0^+}\mathcal{M}(x)=+\infty,  $ then $\alpha(t)=0,   \forall t\in\RR_+.$
\end{lemm}

\begin{Rema}\label{osg23}
In the particular case $\mu(r)=r(1-\log r)$ one can show the following estimate,   see Theorem 5.2.1 \cite{che1}: for every $t\in\RR_+$
$$
a\leq e^{1-\exp{\int_0^t\gamma(\tau)d\tau}}\Longrightarrow \alpha(t)\le a^{\exp{-\int_0^t\gamma(\tau)d\tau}}e^{1-\exp{(-\int_0^t\gamma(\tau)d\tau)}}.
$$

\end{Rema}

\section{Riesz transform and commutators}
\label{sectionRiesz}
A crucial step in the  implementation of  the strategy exposed in the introduction for the proof of Theorem \ref{theo2},    is the study  of  commutators between the Riesz transform  $\mathcal{R}={\partial_{1} }/{  |D|}$ and the convection
 operator $v \cdot \nabla $. The results of this section hold for all space dimension $d\geq 2$.

Let us first recall  some well-known properties of the Riez operator. 
     \begin{prop}
    \label{cor1}  Let  $\mathcal{R}$ be the Riez operator $\mathcal{R}={\partial_{1} }/{  |D|}.$  Then the following hold true.
    
 {\rm (1)}  For every $p \in ]1,   +\infty[,  $ 
    \begin{equation*}
  \label{CZ}
   \| \mathcal{R}  \|_{\mathcal{L}(L^p)} \lesssim 1.
   \end{equation*}

{\rm (2)} Let $\mathcal C$ be a fixed ring. Then,   there exists $\psi\in \mathcal S$
whose spectum does not meet the origin such that 
$$
\mathcal{R} f=2^{dq}\psi(2^{q}\cdot)\star f
$$
for every $f$ with Fourier transform supported  in $2^q\mathcal C$. In particular,  
$\mathcal{R}\Delta_q$ is uniformly bounded (with respct to $q\in \mathbb N$) in $L^p$ for every $p\in [1,  +\infty].$
\end{prop} 
The property 
(1) is  a classical Calder\'on-Zygmund theorem (see \cite{Stein} for instance) and (2) is obvious.

The proof of the next lemma is easy  and can be found in  \cite{HKR1}.
 \begin{lemm}
 \label{commu} Let  $p\in[1,  \infty]$,    $f,  g$ and $h$ be three functions such that $\nabla f\in L^p,   g\in L^\infty$ and $xh\in L^{1}$. Then,  
 $$
 \|h\star(fg)-f(h\star g)\|_{L^p}\leq \|xh\|_{L^{1}}\|\nabla f\|_{L^p}\|g\|_{L^{\infty}}.
 $$
\end{lemm}

\

Here is the  main result of this section.
\begin{Theo}\label{propcom}
Let  $v$ be is  a smooth  divergence-free  vector field. Then the following hold true.

{\rm (1)} For every $s\in]0,  1[$
\begin{eqnarray*}
\|[\mathcal{R},   v]\theta\|_{H^s}\lesssim_s \|\nabla v\|_{L^2}\|\theta\|_{B_{\infty,  2}^{s-1}}+\|v\|_{L^2}\|\theta\|_{L^2},  
\end{eqnarray*}
for every smooth scalar function $\theta$.

{\rm (2)} For every $p\in [2,  \infty]$
$$
\|[\mathcal{R},   v\cdot\nabla]\theta\|_{B_{p,  \infty}^0}\lesssim_p \|\nabla v\|_{L^p}\|\theta\|_{B_{\infty,  \infty}^{0}}+\|v\|_{L^2}\|\theta\|_{L^2},  
$$
for every smooth scalar function $\theta$.

\end{Theo}
 \begin{proof}[Proof of Theorem \ref{propcom}]
 
(1)  We  split the commutator into three parts,   according to  Bony's decomposition
\begin{eqnarray*}
[\mathcal{R},   v]\theta&=&\sum_{q\in\NN}[\mathcal{R},   S_{q-1}v]\Delta_q\theta+\sum_{q\in\NN}[\mathcal{R},   \Delta_qv]S_{q-1}\theta\\
&+&\sum_{q\geq-1} [\mathcal{R},   \Delta_qv]\widetilde{\Delta}_q\theta\\
&=& \sum_{q\in\NN}\mbox{I}_q+\sum_{q\in\NN}\mbox{II}_q+\sum_{q\geq-1}\mbox{III}_q\\
&=&\mbox{I}+\mbox{II}+\mbox{III}.
\end{eqnarray*}
$\bullet$  {\sl Estimation of $\mbox{ I}$}.    According to the point (2) of \mbox{Proposition \ref{cor1}} there  exists  $h\in\mathcal{S}$ whose
 Fourier transform   does not contain the origin such that
 $$
\mbox{I}_q(x)=h_q\star( S_{q-1}v\Delta_q\theta)-S_{q-1}v(h_q\star \Delta_q\theta),  
$$
 where $h_q(x)=2^{dq}h(2^qx)$.
 Applying Lemma \ref{commu} with $p=2$   we infer
  \begin{eqnarray}
  \nonumber
\|\mbox{I}_q\|_{L^2}&\lesssim &\|xh_q\|_{L^1}  \|\nabla S_{q-1} v\|_{L^2}\|\Delta_q\theta\|_{L^\infty}
\\
\label{x1}
&\lesssim& 2^{-q}\|\nabla v\|_{L^2}\|\Delta_q\theta\|_{L^\infty}.
\end{eqnarray}
In the last line we have used the fact that  $\|xh_q\|_{L^1}= 2^{-q}\|xh\|_{L^1}$.
Since  for every $q\in\NN$ the Fourier transform of $\hbox{I}_q$ is supported in a ring of size $2^q$ then 
$$\|\hbox{I}\|_{H^s}^2\simeq  \sum_{q}2^{2qs}\|\hbox{I}_q\|_{L^2}^2.
$$
Combined with \eqref{x1} this yields
\begin{eqnarray*}
\|\hbox{I}\|_{H^s}^2\lesssim\|\nabla v\|_{L^2}^2\|\theta\|_{B_{\infty,  2}^{s-1}}^2.
\end{eqnarray*}
$\bullet$  {\sl Estimation of $\mbox{ II}$}.  As before we can write
$$
\mbox{II}_q(x)=h_q\star( S_{q-1}\theta\Delta_q v)-S_{q-1}\theta(h_q\star \Delta_q v),   
$$ 
and again by Lemma \ref{commu} with $p=2$   we obtain
\begin{eqnarray*}
\|\hbox{II}_q\|_{L^2}&\lesssim &\|\nabla v\|_{L^2}   2^{-q} \|S_{q-1} \theta \|_{L^\infty}
\\
&\lesssim & \|\nabla v\|_{L^2}   2^{-q}\sum_{j\leq q-2} \|\Delta_{j} \theta \|_{L^\infty}.
\end{eqnarray*}
Thus,  
\begin{eqnarray*}
2^{qs}\|\hbox{II}_q\|_{L^2}\lesssim  \|\nabla v\|_{L^2} \big(  2^{(1-s)\cdot}\star 2^{(1-s)\cdot} \|\Delta_{\cdot} \theta \|_{L^\infty}\big)(q).
\end{eqnarray*}
where $\star$ is the discrete convolution in $\mathbb N\cup\{-1\}$. 

Here again the  Fourier transform of $\mbox{II}_{q}$ is   supported in a ring of size $2^q$ and by consequences $ \|\mbox{II}\|_{H^s}\simeq  \big\|2^{s\cdot}\|\hbox{I}_{\cdot}\|_{L^2}\big\|_{\ell^2}$.  Combined with \eqref{dsp} and discrete  Young inequalities (remember  $s<1$) this yields
\begin{eqnarray*}
 \|\mbox{II}\|_{H^s}  &   \lesssim_s & 
 \|\nabla v\|_{L^2}^2 \big\|2^{(s-1)\cdot}\|\Delta_{\cdot}\theta\|_{L^\infty}\big\|_{\ell^2}
  \\
 &\simeq& 
  \|\nabla v\|_{L^2} \|\theta\|_{B_{\infty,  2}^{s-1}}.
 \end{eqnarray*}

 $\bullet$  {\sl Estimation of $\mbox{ III}$}.  We distinguish two parts
 $$
{\rm III}=\sum_{q\geq 1 }[\mathcal R,  \Delta_qv]\widetilde\Delta_q\theta+\sum_{q\le 0}[\mathcal R,  \Delta_qv]\widetilde\Delta_q\theta:=J_1+J_2.
$$
$J_1$ contains   only  terms whose  Fourier transform are  localized  away from zero. For them,     we can use 
 Proposition \ref{cor1} and Lemma \ref{commu} as before. This gives
\begin{eqnarray*}
\|[\mathcal R,  \Delta_qv]\widetilde\Delta_q\theta\|_{L^2}\lesssim 2^{-q}\|\nabla v\|_{L^2}\|\widetilde\Delta_q\theta\|_{L^\infty}.
\end{eqnarray*}
Note that we have used the fact that for $q\geq 1$,   we have   $\|\mathcal{R}\widetilde\Delta_q\theta\|_{L^\infty}\lesssim \|\widetilde\Delta_q\theta\|_{L^\infty}$. 
Now we have
$$
2^{js}\|\Delta_j J_1\|_{L^2}\lesssim\|\nabla v\|_{L^2}\sum_{q\geq j-4}2^{(j-q)s} 2^{q(s-1)}\|\widetilde\Delta_q\theta\|_{L^\infty}
$$
Since $s>0$ then the convolution inequality leads to 
\begin{eqnarray*}
\|J_1\|_{H^s}&\lesssim&\|\nabla v\|_{L^2}\|\theta\|_{B_{\infty,  2}^{s-1}}.
\end{eqnarray*}

$J_2$ contains a finite number of terms with low frequencies and  it  can be  handled  without using the commutator structure. Indeed,   
from  Bernstein inequalities and the  $L^2$-continuity   of  the Riesz transform we  obtain for $q \leq 0$
\begin{eqnarray*}
\|[\mathcal R,  \Delta_qv]\Delta_q\theta\|_{L^2}&\lesssim &
\|\Delta_q v\|_{L^2}\big(\|\Delta_q\theta\|_{L^\infty}+\|\mathcal{R}\Delta_q\theta\|_{L^\infty}  \big)
\\
&\lesssim&\| v\|_{L^2}\big(\|\theta\|_{L^2}+\|\mathcal{R}\theta\|_{L^2}  \big)
\\
&\lesssim&\|v\|_{L^2}\|\theta\|_{L^2}.
\end{eqnarray*}
Thus we get for every $s\in\RR$
$$
\|J_2\|_{H^s}\lesssim\|v\|_{L^2}\|\theta\|_{L^2}.
$$

This ends the proof of (1).

\bigskip

(2) We use again Bony's decomposition to write 
\begin{eqnarray*}
[\mathcal{R},   v\cdot\nabla]\theta&=&\sum_{q\in\NN}[\mathcal{R},   S_{q-1}v\cdot\nabla]\Delta_q\theta+\sum_{q\in\NN}[\mathcal{R},   \Delta_qv\cdot\nabla]S_{q-1}\theta\\
&+&\sum_{q\geq-1} [\mathcal{R},   \Delta_qv\cdot\nabla]\widetilde{\Delta}_q\theta\\
&=&\mbox{I}+\mbox{II}+\mbox{III}.
\end{eqnarray*}
Similarly to \eqref{x1} we have
\begin{eqnarray*}
\big\|[\mathcal{R},   S_{q-1}v\cdot\nabla]\Delta_q\theta\big\|_{L^p}&\lesssim& 2^{-q}\|\nabla v\|_{L^p}\|\Delta_q\nabla\theta\|_{L^\infty}\\
&\lesssim&\|\nabla v\|_{L^p}\|\Delta_q\theta\|_{L^\infty}.
\end{eqnarray*}
Thus we get
$$
\|{\rm I}\|_{B_{p,  \infty}^0}\lesssim\|\nabla v\|_{L^p}\|\theta\|_{B_{\infty,  \infty}^0}.
$$
For the second term we have
\begin{eqnarray*}
\big\|[\mathcal{R},   \Delta_qv\cdot\nabla]S_{q-1}\theta\big\|_{L^p}&\lesssim& 2^{-q}\|\nabla \Delta_qv\|_{L^p}\|S_{q-1}\nabla\theta\|_{L^\infty}\\
&\lesssim&\|\nabla v\|_{L^p}\sum_{j\le q-2}2^{j-q}\|\Delta_q\theta\|_{L^\infty}.
\end{eqnarray*}
It follows  that
$$
\|{\rm II}\|_{B_{p,  \infty}^0}\lesssim\|\nabla v\|_{L^p}\|\theta\|_{B_{\infty,  \infty}^0}.
$$
To estimate the remainder term we use the embedding \mbox{$L^p\hookrightarrow B_{p,  \infty}^0,  $}
\begin{eqnarray*}
\|{\rm III}\|_{B_{p,  \infty}^0}\lesssim \sum_{q\le 1}\| [\mathcal{R},   \Delta_qv\cdot\nabla]\widetilde{\Delta}_q\theta\|_{L^p}+\big\|\sum_{q\geq2} {\rm div }[\mathcal{R},   \Delta_qv]\widetilde{\Delta}_q\theta\big\|_{B_{p,  \infty}^0}.
\end{eqnarray*}
For the first term of the RHS we use Bernstein inequalities ($p\geq2$)  and  the $L^2$ continuity of
 $ \mathcal{R}$ to get 
\begin{eqnarray*}
\sum_{q\le 1}\| [\mathcal{R},   \Delta_qv\cdot\nabla]\widetilde{\Delta}_q\theta\|_{L^p}&\lesssim& \|v\|_{L^2}\sum_{q\le 1}(\|\nabla\widetilde\Delta_q\theta\|_{L^\infty}+\|\mathcal{R}\nabla\widetilde\Delta_q\theta\|_{L^\infty})\\
&\lesssim& \|v\|_{L^2}\|\theta\|_{L^2}.
 \end{eqnarray*}
The second term is estimated as follows
 \begin{eqnarray*}
 \|\sum_{q\geq2} [\mathcal{R},   \Delta_qv\cdot\nabla]\widetilde{\Delta}_q\theta\|_{B_{p,  \infty}^0}&\lesssim&\|\nabla v\|_{L^p}\sup_j\sum_{q\geq j-4}2^{j-q}\|\widetilde\Delta_q\theta\|_{L^\infty}\\
 &\lesssim&\|\nabla v\|_{L^p}\|\theta\|_{B_{\infty,  \infty}^0}.
  \end{eqnarray*}
This ends the proof of Theorem \ref{propcom}.
\end{proof}

\section{Transport-Diffusion models}

\label{sectionTD}

This section contains  some estimates needed  in the proof of Theorem \ref{theo2}. We start with the following 
Besov space estimate for the  transport equation,    for the proof see for example \cite{ah}.
\begin{prop}\label{Besov-pro} Let  $v$  be a  smooth divergence-free vector field. Then,    every scalar solution $\psi $ of the equation
$$
\partial_t\psi+v\cdot\nabla \psi=f,     \quad \psi_{\mid t=0}= \psi^0,  
$$
satifies,   for every $p\in [1,  +\infty]$,   
$$
\|\psi(t)\|_{B_{p,  \infty}^{-1}}\le C \exp\big(C\int^t_0\|v(\tau)\|_{B_{\infty,  1}^1}d\tau\big)\Big(\|\psi^0\|_{B_{p,  \infty}^{-1}}+\int_0^t \|f(\tau)\|_{B_{p,  \infty}^{-1}}d\tau\Big).
$$

\end{prop}


The second proposition  is dedicated  to  some logarithmic  and  $L^p$ estimates. 
\begin{prop}\label{thmlog}  Let $v$  be a  smooth divergence-free vector field,   $\kappa\in\RR_+$ and  $(p,  r)\in[1,  \infty]^2$.  Then  there exists $C>0$,   such that every scalar solution of 
\begin{equation}
\label{adiss}
\partial_t \psi +v\cdot\nabla \psi +\kappa|\DD| \psi =f,   \quad \psi_{\mid t=0}= \psi^0,  
\end{equation}
satisfies
$$
\| \psi \|_{\widetilde L^\infty_tB_{p,  r}^0}\leq C\Big( \|\psi^0\|_{B_{p,  r}^0}+\|f\|_{\widetilde L^1_tB_{p,  r}^0}\Big)\Big( 1+\int_0^t\|\nabla v(\tau)\|_{L^\infty}d\tau \Big),  
$$
and 
$$
\|\psi (t)\|_{L^p}\le\| \psi ^0\|_{L^p}+\int_0^t\|f(\tau)\|_{L^p}d\tau.
$$
\end{prop}
The first result was first proved  by Vishik in \cite{vis} for the case $\kappa=0$ by using
 the special structure of the transport equation. In \cite{H-K2} the first two authors have 
  generalized Vishik result for a transport-diffusion equation where the dissipation
   term takes the form $-\kappa\Delta \psi$.  The method described in \cite{H-K2} 
   can be easily adapted to the model \eqref{adiss},    for more details  
   the complete proof can be found  in \cite{HKR1}. The $L^p$ estimates are proved  in \cite{cc}.

 In the proof of the uniqueness part of the main theorem we shall  also need   
 some estimates for the linearized velocity equation.

\begin{prop}
\label{thm99}
 Let $v$ a  smooth divergence free vector field,   $s\in]-1,  1[$  and $ \rho\in[1,  \infty].$ 
 Let $u$ be a smooth solution of the system
\begin{equation*} 
{\rm (LB)}\qquad\qquad\partial_{t} u+v\cdot\nabla u+\vert\textnormal{D}\vert u+\nabla p=f,  \qquad {\rm div } u=0.
\end{equation*} 
Then,   we have for every $t\in\RR_+$ 
$$
\|u\|_{ L^\infty_t{B}_{2,  \infty}^{s}}\leq Ce^{CV(t)}\Big( \|u^0\|_{{B}_{2,  \infty}^s}+\|f\|_{\widetilde L^\rho_tB_{2,  \infty}^{s-1+\frac1\rho}}(1+t^{1-\frac1\rho}\big)  \Big),  
$$
where $V(t):=\int_0^t\|\nabla v(\tau)\|_{L^\infty}d\tau.$
\end{prop}

\begin{proof}[Proof of Proposition \ref{thm99}]
For $q\in \NN$ we apply the operator $\Delta_q$  to (LB) 
$$
\partial_t u_q+v\cdot\nabla u_q+\vert\textnormal{D}\vert u_q+\nabla p_q =-[\Delta_q,   v\cdot\nabla]u+f_q.
$$
Taking the $L^2$ inner product of this equation with $u_q$ we get due to the incompressibility of $v$ and $u_q$
$$
\frac12\frac{d}{dt}\|u_q\|_{L^2}^2+\int_{\RR^2}(\vert\textnormal{D}\vert u_q ) u_q dx\leq \|u_q\|_{L^2}\Big(\|[\Delta_q,   v\cdot\nabla]u\|_{L^2}+\|f_q\|_{L^2}\Big).
$$
From Parseval identity we get
$$
C2^q\|u_q\|_{L^2}^2\le\int_{\RR^2}(\vert\textnormal{D}\vert u_q ) u_q dx.
$$
Thus we get
$$
\frac{d}{dt}\|u_q(t)\|_{L^2}+c2^q\|u_q(t)\|_{L^2}\leq \|[\Delta_q,   v\cdot\nabla]u(t)\|_{L^2}+\|f_q(t)\|_{L^2}.
$$
Integrating in time this differential inequality we obtain
$$
\|u_q(t)\|_{L^2}\lesssim e^{-ct2^q}\|u_q^0\|_{L^2}+\int_0^te^{-c2^q(t-\tau)}\Big(\|[\Delta_q,   v\cdot\nabla]u(\tau)\|_{L^2}+\|f_q(\tau)\|_{L^2}\Big)d\tau
$$
Using H\"{o}lder inequalities we get
$$
\|u_q\|_{L^\infty_tL^2}\lesssim \|u_q^0\|_{L^2}+\int_0^t\|[\Delta_q,   v\cdot\nabla]u(\tau)\|_{L^2}d\tau+2^{q(-1+\frac1\rho)}\|f_q(\tau)\|_{L^\rho_tL^2}
$$
Multiplying by $2^{qs}$ and taking the supremum over $q\in\NN$ we find
\begin{equation}
\label{rrr}
\sup_{q\in\NN} 2^{qs}\|u_q\|_{L^\infty_tL^2}\lesssim \|u^0\|_{B_{2,  \infty}^s}+\int_0^t\sup_{q\in \NN}2^{qs}\|[\Delta_q,   v\cdot\nabla]u(\tau)\|_{L^2}d\tau+\|f\|_{\widetilde L^\rho_tB_{2,  \infty}^{s-1+\frac1\rho}}.
\end{equation}
Let us recall the following  classical commutator estimate (see \cite{che1} for instance)
$$
\sup_{q\geq-1}2^{qs}\|[\Delta_q,   v\cdot\nabla]u(\tau)\|_{L^2}\lesssim_s\|\nabla v\|_{L^\infty}\|u\|_{B_{2,  \infty}^s},  \qquad \forall s\in(-1,  1).
$$
Combined with \eqref{rrr} this yields
$$
\sup_{q\in\NN} 2^{qs}\|u_q\|_{L^\infty_tL^2}\lesssim \|u^0\|_{B_{2,  \infty}^s}+\int_0^t\|\nabla v(\tau)\|_{L^\infty}\|u(\tau)\|_{B_{2,  \infty}^s}d\tau+\|f\|_{\widetilde L^\rho_tB_{2,  \infty}^{s-1+\frac1\rho}}.
$$
For the low frequency block,    we have from the energy estimate of the localized equation
\begin{eqnarray*}
\frac{d}{dt}\|\Delta_{-1}u(t)\|_{L^2}&\leq &\|[\Delta_{-1},   v\cdot\nabla]u(t)\|_{L^2}+\|\Delta_{-1}f(t)\|_{L^2}\\
&\lesssim&\|\nabla v(t)\|_{L^\infty}\|u(t)\|_{B_{2,  \infty}^s}+\| \Delta_{-1}f(t)\|_{L^2}
\end{eqnarray*}
It follows
\begin{eqnarray*}
\|\Delta_{-1}u(t)\|_{L^2}&\leq &\|\Delta_{-1}u^0\|_{L^2}+\int_0^t\|\nabla v(\tau)\|_{L^\infty}\|u(\tau)\|_{B_{2,  \infty}^s}d\tau+\|\Delta_{-1}f\|_{L^1_tL^2}\\
&\lesssim&\|u^0\|_{B_{2,  \infty}^s}+\int_0^t\|\nabla v(t)\|_{L^\infty}\|u(\tau)\|_{B_{2,  \infty}^s}d\tau+t^{1-\frac1\rho}\| f\|_{\widetilde L^\rho_t B_{2,  \infty}^{s-1+\frac1\rho}}.
\end{eqnarray*}
The outcome is 
$$
\|u(t)\|_{B_{2,  \infty}^s}\lesssim\|u^0\|_{B_{2,  \infty}^s}+\int_0^t\|\nabla v(t)\|_{L^\infty}\|u(\tau)\|_{B_{2,  \infty}^s}d\tau+(1+t^{1-\frac1\rho})\| f\|_{\widetilde L^\rho_t B_{2,  \infty}^{s-1+\frac1\rho}}.
$$
A  Gronwall  inequality gives the claimed result.
\end{proof}

The proof of the next proposition  can be done in a similar way as Theorem 1.2 \mbox{in \cite{ab-Hm}.}
\begin{prop}\label{smooth5}
Let $v$ be a smooth divergence-free vector field and $ \psi  $ be a smooth solution of the equation
$$
\partial_t \psi +v\cdot\nabla \psi +|\DD| \psi =f,  \qquad \psi_{|t=0}=\psi^0.
$$
Then,   for every  $s\in]-1;1[ $ and $  (\rho,   p,  r)\in[1,  \infty]^3$  there exists $C>0$ such that 
 $$
\|\psi\|_{\widetilde L^\infty_tB_{p,  r}^s}\le Ce^{CV(t)}\Big(\|\psi^0\|_{B_{p,  r}^s}+(1+t^{1-\frac1\rho})\| f\|_{\widetilde L^\rho_t B_{p,  r}^{s-1+\frac1\rho}}
  \Big),  \qquad \forall t\in\RR_+,  
$$
where $V(t)=\|\nabla v\|_{L^1_t L^\infty}.$
\end{prop}


\section{Weak solutions}

\label{sectionbouss}

Throughout the coming sections we use the notation $\Phi_k$ to denote any function
of the form 
$$
\Phi_k(t)=  C_{0}\underbrace{ \exp(...\exp  }_{k\,  times}(C_0t)...),  
$$
where $C_{0}$ depends on the involved norms of the initial data and its value may vary from line to line up to some absolute constants. 
We will make an intensive  use (without mentionning it) of  the following trivial facts
$$
\int_0^t\Phi_k(\tau)d\tau\leq \Phi_k(t)\qquad{\rm and}\qquad \exp({\int_0^t\Phi_k(\tau)d\tau})\leq \Phi_{k+1}(t).
$$
In this section we shall establish a   first set of  $\ap$ estimates and   discuss some results about weak solutions
which are easy consequences. These $\ap$ estimates are 
   also needed in order to construct the 
 global strong solutions as  stated in \mbox{Theorem \ref{theo2}.}

The first result  is concerned with weak solutions in  the energy space.
\begin{prop}\label{es-energ}
Let $(v^0,  \theta^0)\in L^2\times L^2,$ then there exists  a  global weak  solution of \eqref{ns} in the space $L^\infty_{\textnormal{loc}}(\RR_+; L^2)\cap L^2_{\textnormal{loc}}(\RR_+; \dot{H}^{\frac12})\times L^\infty(\RR_+;L^2)$  such that 
\begin{eqnarray*}
\|v(t)\|_{L^2}^2+\int_0^t\|v(\tau)\|_{\dot{H}^{\frac12}}^2d\tau&\le& C_0(1+t^2)\\
\|\theta(t)\|_{L^2}&\le&\|\theta^0\|_{L^2}.
\end{eqnarray*}
 If in addition,     $\theta^0 \in L^p$   for some  $p\in[1,  \infty]$,   then there is a weak solution which satisfies
  also  
 $$
 \|\theta(t)\|_{L^p}\le\|\theta^0\|_{L^p}.
 $$
\end{prop}
\begin{proof}[Proof of Proposition \ref{es-energ}]
The  estimate  of $\theta$ in $L^p$ is a consequence  of the incompressibility of the flow.  The $L^2$ energy estimate for the velocity  can be obtained by taking the  $L^2$-inner product of the
  velocity equation in the  Boussinesq system \mbox{with $v$,  }
  \begin{eqnarray*}
  \frac12\frac{d}{dt}\|v(t)\|_{L^2}^2+\int_0^t\|v(\tau)\|_{\dot{H}^{\frac12}}^2d\tau&\le&\|v(t)\|_{L^2}\|\theta(t)\|_{L^2}\\
 &\leq& \|v(t)\|_{L^2}\|\theta^0\|_{L^2}.
  \end{eqnarray*}
  Thus  we obtain 
\begin{eqnarray*}
\|v(t)\|_{L^2}&\le&\|v^0\|_{L^2}+\int_0^t\|\theta(\tau)\|_{L^2}d\tau\\
&\le&\|v^0\|_{L^2}+\|\theta^0\|_{L^2} t.
\end{eqnarray*}
Inserting this inequality into the previous one leads to the desired estimate. Now to construct global solution we can proceed in a classical way using  the Friedrichs method. We omit here the details and for complete description of this method  we refer for example  to \cite{dp1}.
\end{proof}
\begin{Rema}
Due to the weak regularity of the velocity,   the uniqueness problem for these weak solutions seems
 an interesting widely open problem. 
\end{Rema}
Now we aim at constructing global weak solutions for more regular initial data near the scaling space $\dot{H}^1\times L^2$ described in the introduction of this paper.
\begin{prop}
\label{prop500}
Let $(v^0,  \theta^0)\in H^1\times L^2\cap L^r,  $ with $r\in]4,  \infty].$ Then,  
 there exists a global weak solution  $(v,  \theta)$ for  the system \eqref{ns}  such that
$$
\|\omega(t)\|_{L^2}^2+\int_0^t\|(\omega-\mathcal{R}\theta)(\tau)\|_{\dot{H}^{\frac12}}^2d\tau \leq \Phi_1(t),  
$$
where $\omega= \mbox{curl }v$.

\end{prop}
\begin{Rema}
From the estimate of Proposition \ref{prop500} we see that we have an additional   smoothing effect for the quantity $\omega-\mathcal{R}\theta$ since $\theta$ belongs only to the space $L^\infty_t(L^2\cap L^r)$. This phenomenon illustrates the strong coupling between the velocity and the temperature.
\end{Rema}
\begin{proof}[Proof of Proposition \ref{prop500}] We shall,   here again,   restrict ourselves  to the proof of  the {\it a priori} estimates.  The construction of global solutions can be done by following \cite{dp1}.  As explained  in the introduction we do not have any available obvious  $L^p$ estimate for the vorticity. Thus in order to get some $L^p$ estimates we write the Boussinesq  system under its diagonal form.  For this purpose we set $\Gamma=\omega-\mathcal{R}\theta.$ Then we get from \eqref{ns}
\begin{equation}\label{eq-com}
(\partial_t+v\cdot\nabla+|\DD|)\Gamma=[\mathcal{R},  v\cdot\nabla]\theta.
\end{equation}
By using  the identity  $[\mathcal{R},  v\cdot\nabla]\theta=\textnormal{div}\big([\mathcal{R},  v]\theta\big)$ and a  standard $L^2$ energy estimate,   we find 
\begin{eqnarray*}
\frac12\frac{d}{dt}\|\Gamma(t)\|_{L^2}^2+\|\Gamma(t)\|_{\dot{H}^{\frac12}}^2&=&\int_{\RR^2}\textnormal{div}\big([\mathcal{R},  v]\theta\big)(t,  x)\Gamma(t,  x)dx\\&\le& \big\|[\mathcal{R},  v]\theta(t)\big\|_{\dot{H}^{\frac12}}\|\Gamma(t)\|_{\dot{H}^{\frac12}}.
\end{eqnarray*}
From Theorem  \ref{propcom} and Proposition \ref{es-energ} we have 
\begin{eqnarray*}
\big\|[\mathcal{R},  v]\theta(t)\big\|_{\dot{H}^{\frac12}}
&\lesssim&
\|\nabla v(t)\|_{L^2}\|\theta(t)\|_{B_{\infty,  2}^{-\frac12}}+\|v(t)\|_{L^2}\|\theta(t)\|_{L^2}\\
&\lesssim& \|\omega(t)\|_{L^2}\|\theta(t)\|_{L^r}+(1+t).
\\
&\lesssim& \|\omega(t)\|_{L^2}\|\theta^0\|_{L^r}+(1+t).
\end{eqnarray*}
Here,   we have used that $\|\nabla v \|_{L^2} \approx \|\omega \|_{L^2}$ and the continuous embedding  \mbox{$L^r\hookrightarrow B_{\infty,  2}^{-\frac12},  $} for every $r>4$. Thus we get
\begin{eqnarray*}
\big\|[\mathcal{R},  v]\theta(t)\big\|_{\dot{H}^{\frac12}}\lesssim\|\omega(t)\|_{L^2}+ (1+t).
\end{eqnarray*}
However,   the $L^2$-continuity of  $\mathcal{R}$ and the conservation of the $L^2$ norm of $\theta$ yield together  
\begin{eqnarray*}
 \|\omega(t)\|_{L^2}&\le&\|\Gamma(t)\|_{L^2}+\|\mathcal{R}\theta(t)\|_{L^2}\\
 &\lesssim&\|\Gamma(t)\|_{L^2}+\|\theta^0\|_{L^2}.
\end{eqnarray*}
Collecting  the previous  estimates, we find
\begin{equation*}
\frac12\frac{d}{dt}\|\Gamma(t)\|_{L^2}^2+\|\Gamma(t)\|_{\dot{H}^{\frac12}}^2\lesssim \big( \|\Gamma(t)\|_{L^2}+1+t \big) \|\Gamma(t)\|_{\dot{H}^{\frac12}}.
\end{equation*}
It follows from  the Young inequality that 
\begin{equation*}
\frac{d}{dt}\|\Gamma(t)\|_{L^2}^2+\|\Gamma(t)\|_{\dot{H}^{\frac12}}^2\le C_{0} \|\Gamma(t)\|_{L^2}^2+C_{0}(1+t^2).
\end{equation*}
After an Integration  in  time  and the use of the  Gronwall inequality we obtain
\begin{eqnarray*}
\nonumber\|\Gamma(t)\|_{L^2}^2+\int_0^t\|\Gamma(\tau)\|_{\dot{H}^{\frac12}}^2d\tau&\le& C_0\big(1+ t^2\big)
e^{ C_0t}\\
&\le& \Phi_1(t).
\end{eqnarray*}

This ends the proof of the claimes  {\it a priori} estimate in  Proposition \ref{prop500}.

\end{proof}
\section{Proof of Theorem \ref{theo2}}
\label{proof}
The proof of Theorem \ref{theo2} will be done in three steps. First we prove other   {\it {\it a priori }   } estimates for the equations  \eqref{ns}. Second,  we prove the uniqueness part. Finally,     we   discuss    the construction of the solutions. 
 
\subsection{A priori estimates}
We have already proven some {\it a priori} estimates in Proposition \ref{es-energ} and Proposition \ref{prop500}.
In the sequel  we shall prove some other estimates. 
To estimate the  $L^p$ norm of $\omega$ we have to distinguish two cases. If $p\in ]2,  4[$  then 
we can prove this estimate in one step. If not,    we establish  first an estimate of   the $L^r$ norms of 
$\omega $ for all $r\in ]2,  4[$ and then\footnote{ Notice that,   concerning the velocity,   the 
assumptions of Theorem  \ref{theo2} are equivalent to $v^0\in L^2$ and $\omega^0\in L^r$ for 
all $r\in [2,  p],  $ for some $p\in ]2,  +\infty[$.} use it to get a smoothing effect on 
the quantity $(\omega-\mathcal{R}\theta)$. This will in turn  yield the crucial  estimate on  the Lipschitz 
norm of the velocity. Finally,   
by using the estimate on the lipschitz norm of the velocity,   one can easily propagate  the
 $L^p$ norm of the vorticity and conclude the subsection about the  $\ap$ estimates. 
\begin{prop}
\label{prop600}
Let $(v,  \theta)$ be a  solution of the Boussinesq 
system \eqref{ns}  such that 
 \mbox{$v^0\in H^1\cap \dot W^{1,  p}$}   and $\theta^0\in L^2\cap L^\infty$ 
  with $p\in]2,  +\infty[.$ Then,    
$$
\|\omega(t)\|_{L^r}^r+\int_0^t\|(\omega-\mathcal{R}\theta)(\tau)\|_{L^{2r}}^rd\tau\leq \Phi_1(t),  
$$
for every  $r\in [2,  4[\cap [2,  p].$
\end{prop}

\begin{proof}[Proof of Proposition \ref{prop600}]
 Multiplying  \eqref{eq-com}  by $|\Gamma|^{r-2}\Gamma$ and integrating in space
  variable we get for every  $s\in]0,  1[$ ($s$ will be carefully chosen later)
\begin{eqnarray}\label{eqt2}
\nonumber\frac1r\frac{d}{dt}\|\Gamma(t)\|_{L^r}^r+\int_{\RR^2}(|\DD|\Gamma)\,   
|\Gamma|^{r-2}\Gamma dx&=&\int_{\RR^2}\textnormal{div}\big([\mathcal{R},  v]\theta\big)|
\Gamma|^{r-2}\Gamma dx\\&\le& \big\|[\mathcal{R},  v]\theta(t)\big\|_{\dot{H}^{1-s}}\||
\Gamma|^{r-2}\Gamma(t)\|_{\dot{H}^{s}}.
\end{eqnarray}
According  to Lemma 3.3 in \cite{Ju} one has 
$$
\frac2r\| |\Gamma|^{\frac r2}\|_{\dot{H}^{\frac12}}^2\le \int_{\RR^2}(|\DD|\Gamma)\,     |\Gamma|^{r-2}\Gamma dx.
$$
Combining this estimate with  the  Sobolev embedding $\dot{H}^{\frac12}\hookrightarrow L^4,    $  we find 
\begin{equation}\label{eqt3}
\| \Gamma\|_{L^{2r}}^r\lesssim \int_{\RR^2}(|\DD|\Gamma)\,   |\Gamma|^{r-2}\Gamma dx.
\end{equation}
To estimate the RHS of \eqref{eqt2},   we use the following lemma (see the appendix for the proof).
\begin{lemm}
\label{propDelta}
\item Let    $\beta\in [2,  +\infty[$ and $s\in]0,  1[$.Then,   we have \begin{equation*}
\||u|^{\beta-2}u\|_{\dot{H}^s}\lesssim\|u\|_{L^{2\beta}}^{\beta-2}\|u\|_{\dot{H}^{s+1-\frac2\beta}},  
\end{equation*}
 for every smooth function $u.$
\end{lemm}
Combined with   (\ref{eqt2}) and (\ref{eqt3}) this lemma yields
$$
\frac{d}{dt}\|\Gamma(t)\|_{L^r}^r+c\|\Gamma(t)\|_{L^{2r}}^r \lesssim \big\|[\mathcal{R},  v]\theta(t)\big\|_{\dot{H}^{1-s}} 
\|\Gamma\|_{L^{2r}}^{r-2}\|\Gamma\|_{\dot{H}^{s+1-\frac2r}}.
$$
We choose $s\in]0,  1[$ such that $s+1-\frac2r=\frac12$ which means that $s=\frac2r-\frac12,  $ this is
 possible if $r\in[2,  4[.$ Thus we get
$$
\frac{d}{dt}\|\Gamma(t)\|_{L^r}^r+c\|\Gamma(t)\|_{L^{2r}}^r \lesssim \big\|[\mathcal{R},  v]
\theta(t)\big\|_{\dot{H}^{1-s}} \|\Gamma\|_{L^{2r}}^{r-2}\|\Gamma\|_{\dot{H}^{\frac12}}.
$$
Theorem  \ref{propcom} and Proposition  \ref{prop500}  yield
\begin{eqnarray*}
\big\|[\mathcal{R},  v]\theta(t)\big\|_{\dot{H}^{1-s}}&\le&\big\|[\mathcal{R},  v]\theta(t)\big\|_{{H}^{1-s}}\\
&\lesssim&
 \|\nabla v\|_{L^2}\|\theta\|_{{B}_{\infty,  2}^{-s}}+\|v(t)\|_{L^2}\|\theta(t)\|_{L^2}\\
&\lesssim&\|\omega(t)\|_{L^2}\|\theta(t)\|_{{B}_{\infty,  2}^{-s}}+C_0(1+t^2)\\
&\le& \Phi_1(t)\|\theta(t)\|_{{B}_{\infty,  2}^{-s}}+C_0(1+t^2).
\end{eqnarray*}
It suffices now to use the  embedding $L^\infty\hookrightarrow {B}_{\infty,  2}^{-s}$ for $s>0.$ 
  and Proposition \ref{es-energ} to get  that
$$
\big\|[\mathcal{R},  v]\theta(t)\big\|_{\dot{H}^{1-s}}\le \Phi_1(t).
$$
Therefore
$$
\frac{d}{dt}\|\Gamma(t)\|_{L^r}^r+c\|\Gamma(t)\|_{L^{2r}}^r \le \Phi_1(t) \|\Gamma\|_{L^{2r}}^{r-2}\|\Gamma\|_{\dot{H}^{\frac12}}.
$$
From the  Young inequality 
$$
|ab|\leq C|a|^{\frac{r}{2}}+\frac c2 |b|^{\frac{r}{r-2}}
$$
we get
$$
\frac{d}{dt}\|\Gamma(t)\|_{L^r}^r+\|\Gamma(t)\|_{L^{2r}}^r \le \Phi_1(t)\|\Gamma(t)\|_{\dot{H}^{\frac12}}^{\frac r2}.
$$
 After integration in time and the application  of the   H\"older inequality (remember that $r<4$), we find 
 \begin{eqnarray}
 \nonumber
\|\Gamma(t)\|_{L^r}^r+\int_0^t\|\Gamma(\tau)\|_{L^{2r}}^rd\tau&\le& \Phi_1(t)\big(\int_0^t\|\Gamma(\tau)\|_{\dot H^{\frac12}}^2d\tau\big)^{r/4}+\|\Gamma^0\|_{L^r}^r
\\
 \label{eqt1}
&\le& \Phi_1(t).
\end{eqnarray}
In the last line we have used Proposition \ref{prop500}.
\end{proof}
Next we give a smoothing effect for the quantity $\Gamma$ which will be the keystone of a Lipschitz control of the velocity.
\begin{prop}
\label{gggg}
Under the assumptions of  Proposition \ref{prop600} we have 
$$
\|\omega-\mathcal{R}\theta\|_{\widetilde L^\rho_tB_{r,  1}^{\frac 2r}}\leq  \Phi_1(t),  
$$
for every \mbox{$r\in [2,  4[\cap [2,  p]$} and $\rho\in [1,  \frac r2[.$
\end{prop}

\begin{proof}[Proof of Proposition \ref{gggg}]
For $q\in\NN$ we set $\Gamma_q=\Delta_q\Gamma.$ Then,   we localize   in frequencies the equation \eqref{eq-com}  for
 $\Gamma$ to get 
\begin{eqnarray*}
\partial_t\Gamma_q+v\cdot\nabla\Gamma_q+|\DD|\Gamma_q&=&-[\Delta_q,  v\cdot\nabla]\Gamma+\Delta_q([\mathcal{R},  v\cdot\nabla]\theta)\\
&:=&f_q
\end{eqnarray*}
Since $\Gamma_q$ is a real-valued  function  then multiplying the above equation by $|\Gamma_q|^{r-2}\Gamma_q$ and integrating in the  space variable we find
$$
\frac1r\frac{d}{dt}\|\Gamma_q(t)\|_{L^r}^r+\int_{\RR^2}(|\DD|\Gamma_q)\,   |\Gamma_q|^{r-2}\Gamma_q dx\le\|\Gamma_q(t)\|_{L^r}^{r-1}\|f_q(t)\|_{L^r}.
$$
From \cite{cmz}, we have  the following generalized Bernstein inequality
$$ \int_{\mathbb{R}^d}( |\DD| \Gamma_{q})  |\Gamma_{q}|^{r-2} \Gamma_{q}\, dx \geq c 2^q  \|\Gamma_{q}\|_{L^r}^r,
$$ 
for some $c>0$ independent of $q$ and hence we find
$$
\frac1r\frac{d}{dt}\|\Gamma_q(t)\|_{L^r}^r+c2^{q}\|\Gamma(t)\|_{L^r}^r\le\|\Gamma_q(t)\|_{L^r}^{r-1}\|f_q(t)\|_{L^r}.
$$
This yields
$$
\|\Gamma_q(t)\|_{L^r}\le e^{-ct2^q}\|\Gamma_q^0\|_{L^r}+\int_0^te^{-c(t-\tau)2^q}\|f_q(\tau)\|_{L^r}d\tau.
$$
By taking  the $L^\rho[0,  t]$ norm  and by  using   convolution inequalities,   we find 
\begin{eqnarray}
\label{trois0}
\nonumber2^{q\frac2r}\|\Gamma_q\|_{L^\rho_tL^r}\lesssim 2^{q(\frac2r-\frac1\rho)}\|\Gamma_q^0\|_{L^r}&+&2^{q(\frac2r-\frac1\rho)}\int_0^t\big\|[\Delta_q,   v\cdot\nabla]\Gamma(\tau)\big\|_{L^r}d\tau\\
&+&2^{q(\frac2r-\frac1\rho)}\int_0^t\big\|\Delta_q([\mathcal{R},  v\cdot\nabla]\theta)(\tau)\big\|_{L^r}d\tau.
\end{eqnarray}
To estimate the second integral of the RHS we use the  part  (2) of  Theorem  \ref{propcom},   Proposition \ref{es-energ}
 and Proposition \ref{prop600} to get,   for every $q\in \mathbb N$
\begin{eqnarray}
\nonumber
\big\|\Delta_q([\mathcal{R},  v\cdot\nabla]\theta)\big\|_{L^r}&\lesssim&\|\nabla v\|_{L^r}\|\theta^0\|_{L^\infty}+\|v\|_{L^2}\|\theta\|_{L^2}\\
\nonumber
&\lesssim&\|\omega\|_{L^r}\|\theta^0\|_{L^\infty}+\|v\|_{L^2}\|\theta\|_{L^2}\\ 
\label{sss}
&\le&\Phi_1(t).
\end{eqnarray}

To estimate the first integral of the RHS we use 
 the following lemma (see the appendix for the proof).
 \begin{lemm}
 \label{comDelta1}
 Let $v$ be a smooth divergence-free vector field and $f$ be a smooth scalar function. Then,    for all $ \alpha\in[1,  \infty]$ and $ q\geq -1,  $
$$
\|[\Delta_q,   v\cdot\nabla]f\|_{L^\alpha}\lesssim \|\nabla v\|_{L^r}\|f\|_{B_{\alpha,  1}^{ {2 \over \alpha} } }.
$$
\end{lemm}
Using this lemma and Proposition \ref{prop600} we infer
\begin{eqnarray*}
\big\|[\Delta_q,   v\cdot\nabla]\Gamma\big\|_{L^r}&\lesssim& \|\nabla v\|_{L^r}\|\Gamma\|_{B_{r,  1}^{\frac2r}}\\
&\lesssim&
 \|\omega\|_{L^r}\|\Gamma\|_{B_{r,  1}^{\frac2r}}
 \\
&\le&\Phi_1(t)\|\Gamma\|_{B_{r,  1}^{\frac2r}}.
\end{eqnarray*}
Let $N\in \Bbb N$  to be chosen later. We shall split the sum that we have to estimate into two parts $q<N$ and $q\geq N$.
Since  $4>r>2\rho$ we get 
\begin{eqnarray}
\label{un}
\sum_{q\geq N}2^{q(\frac2r-\frac1\rho)}\big\|[\Delta_q,   v\cdot\nabla]\Gamma\big\|_{L^r}\le 2^{-N(\frac1\rho-\frac2r)}\Phi_1(t)\|\Gamma\|_{B_{r,  1}^{\frac2r}}.
\end{eqnarray}
To estimate the low frequencies,   we first  use the crude estimate
\begin{eqnarray}
\nonumber
 \sum_{q<N}2^{q\frac2r}\|\Gamma_q\|_{L^\rho_tL^r}  &\lesssim&
 2^{\frac2rN}  \|\Gamma\|_{L^\rho_t L^r}
 \\
 \label{cinq}
 &\le& 2^{\frac2rN}\Phi_1(t).
 \end{eqnarray}
 In the last line we have used \eqref{eqt1}.
Gathering   \eqref{trois0},   \eqref{sss},    \eqref{un} and \eqref{cinq} together yield
\begin{eqnarray*}
\nonumber
\|\Gamma\|_{\widetilde L^\rho_tB_{r,  1}^{\frac2r}}
&=&\sum_{q<N}2^{q\frac2r}\|\Delta_q\Gamma\|_{L^\rho_t L^r}+\sum_{q\geq N}2^{q\frac2r}\|\Delta_q\Gamma\|_{L^\rho_t L^r}\\
& \le& 2^{\frac2rN} \Phi_1(t)+ 2^{-N(\frac1\rho-\frac2r)} \Phi_1(t)\|\Gamma\|_{ L^1_tB_{r,  1}^{\frac2r}}
\\
 &\le& 2^{\frac2rN} \Phi_1(t)+ 2^{-N(\frac1\rho-\frac2r)} \Phi_1(t)t^{1-\frac1\rho}\|\Gamma\|_{  L^\rho_tB_{r,  1}^{\frac2r}}
 \\
 &\le& 2^{\frac2rN} \Phi_1(t)+ 2^{-N(\frac1\rho-\frac2r)} \Phi_1(t)\|\Gamma\|_{ \widetilde L^\rho_tB_{r,  1}^{\frac2r}}
 \end{eqnarray*}

We choose $N$ such that 
$$
2^{-N(\frac1\rho-\frac2r)} \Phi_1(t)\approx \frac{1}{2}
$$
and we finally obtain
$$
\|\Gamma\|_{\widetilde L^\rho_tB_{r,  1}^{\frac2r}} \leq \Phi_1(t).
$$
This ends  the proof of the desired result.
\end{proof}
\begin{Rema}
From the embedding $B_{r,  1}^{\frac2r}\hookrightarrow B_{\infty,  1}^0$ we immediately get from
 the above estimate that  for $t\in\RR_+$
\begin{equation}\label{smooth1}
\|\Gamma\|_{\widetilde L^\rho_tB_{\infty,  1}^0}\le \Phi_1(t).
\end{equation}
\end{Rema}

\bigskip

  The previous propositions allow us to prove a crucial $\ap$  estimate on the gradient of $v$ and  to 
  propagate the Besov norm of  $\theta$.
 This  will be  particularly  important  for the uniqueness part of the proof of the main theorem. 
\begin{prop}
\label{prons1} 
Let $(v,  \theta)$ be a smooth solution of the  system \eqref{ns}. Let \mbox{$v^0\in H^1\cap \dot W^{1,  p}$} 
 with $p\in]2,  +\infty[$   and $\theta^0\in L^2\cap B_{\infty,  1}^0$. Then,   we have  
$$
\|v\|_{\widetilde L^\rho_t B_{\infty,  1}^1}+\|\theta (t)\|_{B_{\infty,  1}^{0}}\leq  \Phi_1(t),  
$$
 for every $\rho\in [1,  2[\cap [1,  \frac p2[.$ 
\end{prop}
\begin{proof}[Proof of Proposition \ref{prons1}]
 Using the definition of $\Gamma$ and \eqref{smooth1} for $\rho=1$ we get
\begin{eqnarray}
\label{tite0}
\nonumber\|\omega\|_{ L^1_tB_{\infty,  1}^0}&\leq&\|\Gamma\|_{ L^1_tB_{\infty,  1}^0}+\|\mathcal{R}\theta\|_{ L^1_tB_{\infty,  1}^0}\\
&\le&\Phi_1(t)+\|\mathcal{R}\theta\|_{L^1_tB_{\infty,  1}^0}.
\end{eqnarray}
Now from Bernstein inequality,   Proposition \ref{cor1} and Proposition  \ref{es-energ},   we find
\begin{eqnarray}
\nonumber
\nonumber\|\mathcal{R}\theta(t)\|_{B_{\infty,  1}^0}&\le&\|\Delta_{-1}\mathcal{R}\theta(t)\|_{L^\infty}+\|\theta(t)\|_{ B_{\infty,  1}^0}
\\
\label{tite1}
&\lesssim& \|\theta^0\|_{L^2}+\|\theta(t)\|_{B_{\infty,  1}^0}.
\end{eqnarray}
By applying  Proposition \ref{thmlog} to the second equation of \eqref{ns} we get 
\begin{eqnarray}
\label{tite2}
\|\theta\|_{\widetilde L^\infty_tB_{\infty,  1}^0}\le \|\theta^0\|_{B_{\infty,  1}^0}\Big(1+\|\nabla v\|_{L^1_tL^\infty} \Big).
\end{eqnarray}
However,   Bernstein inequality and the estimate $2^q\| \Delta_qv\|_{L^\infty}\approx\|\Delta_q \omega\|_{L^\infty}$ for very $q\in\NN$ yield together
\begin{eqnarray}
\nonumber
\|v\|_{L^1_tB_{\infty,  1}^1} &\lesssim&  \|\Delta_{-1}v\|_{L^1_tL^\infty}+\|\omega\|_{L^1_tB_{\infty,  1}^0}
\\ 
\nonumber&\lesssim& \|v\|_{L^1_tL^2}+\|\omega\|_{L^1_tB_{\infty,  1}^0}
\\
\label{tite3}
&\le&C_0(1+t^2)+C_0\int_0^t\|\omega(\tau)\|_{B_{\infty,  1}^0}d\tau.
\end{eqnarray}
In the last line we have used Proposition \ref{es-energ}. 

We set $\displaystyle{X(t)=\int_0^t\|\omega(\tau)\|_{B_{\infty,  1}^0}d\tau}.$
Combining  \eqref{tite0},   \eqref{tite1},   \eqref{tite2} and \eqref{tite3}  we get
\begin{eqnarray*}
X(t)
\le \Phi_1(t)+C_0\int_0^tX(\tau)d\tau.
\end{eqnarray*}
 and the Gronwall inequality  yields
$$
\int_0^t\|\omega(\tau)\|_{B_{\infty,  1}^0}d\tau\le \Phi_1(t).
$$
Form the above  estimate and    \eqref{tite3} we  also get 
\begin{eqnarray*}
\nonumber
\| v\|_{L^1_tB_{\infty,  1}^1} &\lesssim& C_0(1+t^2)+  C_0X(t)
\\ 
\nonumber
&\le&\Phi_1(t).
\end{eqnarray*}
By combining this estimate with \eqref{tite1},   we find
\begin{equation*}
\|\theta\|_{\widetilde L^\infty_tB_{\infty,  1}^0}\le \Phi_1(t).
\end{equation*}
This leads to
\begin{equation}
\label{tites1}
\|\mathcal{R}\theta\|_{\widetilde L^\infty_tB_{\infty,  1}^0}\le \Phi_1(t).
\end{equation}
Finally \eqref{smooth1} yields for  every $\rho\in [1,  \frac p2[$
\begin{eqnarray*}
\nonumber\|\omega\|_{\widetilde L^\rho_tB_{\infty,  1}^0}&\leq&\|\Gamma\|_{\widetilde L^\rho_tB_{\infty,  1}^0}+\|\mathcal{R}\theta\|_{\widetilde L^\rho_tB_{\infty,  1}^0}\\
&\le&\Phi_1(t)+\|\mathcal{R}\theta\|_{\widetilde L^\rho_tB_{\infty,  1}^0}.
\end{eqnarray*}
By using H\"older inequality and \eqref{tites1} we find
\begin{eqnarray*}
\|\mathcal{R}\theta\|_{\widetilde L^\rho_tB_{\infty,  1}^0} &\lesssim& t^{\frac1\rho}\|\mathcal{R}\theta\|_{ \widetilde L^\infty_tB_{\infty,  1}^0}
\\
&\leq&\Phi_1(t).
\end{eqnarray*}
It follows that
$$
\|\omega\|_{\widetilde L^\rho_tB_{\infty,  1}^0}\le\Phi_1(t),  
$$
and then 
$$
\|v\|_{\widetilde L^\rho_tB_{\infty,  1}^1}\le\Phi_1(t).
$$

\end{proof}

It remains finally  to propagate  the  $L^p$  norm of the vorticity (when $p\geq 4$).
\begin{prop}
\label{lpnorm}
Under the hypotheses  of \mbox{Proposition \ref{prons1}}  we have $$
\|\omega(t)\|_{L^p}\le\Phi_2(t),  
$$
for every  $t\in\RR_+.$
\end{prop}
\begin{proof}[Proof of Proposition \ref{lpnorm}]
Recall that the quantity $\Gamma=\omega-\mathcal{R}\theta$ satisfies
$$
\partial_t\Gamma+v\cdot\nabla\Gamma+|\DD|\Gamma=[\mathcal{R},  v\cdot\nabla]\theta.
$$
Using  Proposition \ref{thmlog} we find
$$
\|\Gamma(t)\|_{L^p}\leq\|\Gamma^0\|_{L^p}+\int_0^t\big\|[\mathcal{R},  v\cdot\nabla]\theta(\tau)\big\|_{L^p}d\tau.
$$
Recall now the following  commutator result proven in \cite{HKR1}: for $p\in[2,  \infty[$ we have 
$$
\big\|[\mathcal{R},  v\cdot\nabla]\theta\big\|_{B_{p,  1}^0}\lesssim\|\nabla v\|_{L^p}\big(\|\theta\|_{B_{\infty,  1}^0}+\|\theta\|_{L^p}\big)
$$
It follows from the  Calder\`on-Zygmund  Theorem  and Proposition \ref{prons1} that 
$$
\big\|[\mathcal{R},  v\cdot\nabla]\theta(t)\big\|_{B_{p,  1}^0}\lesssim\Phi_1(t)\|\omega(t)\|_{L^p}.
$$
On the other hand we have
\begin{eqnarray*}
\|\omega(t)\|_{L^p}&\le& \|\Gamma(t)\|_{L^p}+\|\mathcal{R}\theta(t)\|_{L^p}\\
&\lesssim& \|\Gamma(t)\|_{L^p}+\|\theta^0\|_{L^p}.
\end{eqnarray*}
By putting together these estimates,   we find 
$$
\|\omega(t)\|_{L^p}\leq\|\omega^0\|_{L^p}+\|\theta^0\|_{L^p}+\int_0^t\Phi_1(\tau)\|\omega(\tau)\|_{L^p}d\tau.
$$
It suffices  to use  the Gronwall inequality to end the proof.
\end{proof}
\begin{Rema} Taking this estimate  into account,   we can also trivially extend  the results of 
 Proposition \ref{gggg} and Proposition \ref{prons1} to every $\rho\in [1,  \frac p2[$.
\end{Rema}


\subsection{Uniqueness}

We  will now  prove a uniqueness result for the system  \eqref{ns}  in the following  space 
$$
\mathcal{X}_T:=L^\infty_TH^1\cap  L^1_T{B_{\infty,  1}^1}\times  L^\infty_T\big(L^2\cap B_{\infty,  1}^0\big).
$$
Let $(v^i,  \theta^i) $ two solutions of the system \eqref{ns} with initial data $(v_i^0,  \theta_i^0)$ 
and lying in the space $\mathcal{X}_T$. We set $v=v^1-v^2$,   $\theta=\theta^1-\theta^2$. Then
\begin{eqnarray*}
\partial_t v+v^2\cdot\nabla v+|\DD| v+\nabla p&=&-v\cdot\nabla v^1+\theta e_2\\
\partial_t\theta+v^2\cdot\nabla \theta&=&-v\cdot\nabla\theta^1.
\end{eqnarray*}
To estimate $v$,   we shall use Proposition  \ref{thm99}.
 By considering the equation for $v$ as a linear equation with a right hand-side which is made
  of the sum of  two terms,   we can write $v=V_{1}+ V_{2}$ where $V_{i}$ solves
  $$ \partial_t V_{i}+v^2\cdot\nabla V_{i}+|\DD| V_{i}+\nabla p_{i}= F_{i}$$
   with $F_{1}= - v \cdot \nabla v^1,   $ and $F_{2}  =  \theta e_{2}$.
   To estimate $V_{1}$,   we use Proposition \ref{thm99} for  $\rho=1$  and $s=0$
    while to estimate $V_{2}$,   we use  Proposition \ref{thm99} for $\rho=+\infty$,   $s=0$.
This yields  for every $t\in [0,  T]$
\begin{equation}
\label{log34}
\|v(t)\|_{B_{2,  \infty}^0}\lesssim  
e^{CV_2(t)}\Big(\|v^0\|_{B_{2,  \infty}^0}+\|v\cdot\nabla v^1\|_{L^1_tB_{2,  \infty}^{0}}+
\|\theta\|_{ L^\infty_tB_{2,  \infty}^{-1}}(1+t) \Big).
\end{equation}
From Lemma \ref{lemproduit} in  the appendix we get
\begin{eqnarray*}
 \|v\cdot\nabla v^1\|_{B_{2,  \infty}^{0}}\lesssim\|v^1\|_{B_{\infty,  1}^1}\|v \|_{L^2}.
\end{eqnarray*}
Now,   by  using  the logarithmic interpolation inequality of Lemma \ref{log-inter} combined with easy computations,  
 we find 
\begin{eqnarray*}
\|v\|_{L^2}&\lesssim& \|v\|_{B_{2,  \infty}^0}\log\Big(e+\frac{\|v\|_{H^1}}{\|v\|_{B_{2,  \infty}^0}}\Big)\\
&\lesssim& \|v\|_{B_{2,  \infty}^0}\log\Big(e+\frac{1}{\|v\|_{B_{2,  \infty}^0}}\Big) \log\big(e+\|v\|_{H^1}\big).
\end{eqnarray*}
Thus we get
\begin{eqnarray}
\label{55}
 \|v\cdot\nabla v^1\|_{B_{2,  \infty}^{0}}\lesssim \|v^1\|_{B_{\infty,  1}^1}\log\big(e+\|v\|_{H^1}\big)\mu(\|v\|_{B_{2,  \infty}^0}) .
\end{eqnarray}
where $\mu(x)=x\log(e+1/x).$
On the other hand,   applying  Proposition \ref{Besov-pro} with $p=2$ to    $\theta$ yields
\begin{equation}
\label{thetaunique2}
\|\theta\|_{ L^\infty_tB_{2,  \infty}^{-1}}\lesssim e^{C\|v^2\|_{L^1_tB_{\infty,  1}^1}}\Big(\|\theta^0\|_{B_{2,  \infty}^{-1}}+\int_0^t\|v\cdot\nabla\theta^1(\tau)\|_{B_{2,  \infty}^{-1}}d\tau\Big).
\end{equation}
To estimate the right hand-side,   we use the following product estimate (see Lemma \ref{lemproduit} in  the appendix)
\begin{equation*}
\|v\cdot\nabla\theta^1\|_{B_{2,  \infty}^{-1}}\lesssim \|v\|_{L^2}\|\theta^1\|_{B_{\infty,  1}^{0}}.
\end{equation*}
The combination of  this estimate with Lemma \ref{log-inter} yield
\begin{equation}
\label{thetaunique1}
\|v\cdot\nabla\theta^1\|_{B_{2,  \infty}^{-1}}\lesssim \|\theta^1\|_{B_{\infty,  1}^{0}}
\log\big(e+\|v\|_{H^1}\big)\mu(\|v\|_{B_{2,  \infty}^0}).
\end{equation}
We set $X(t)=\|\theta\|_{ L^\infty_tB_{2,  \infty}^{-1}}+\|v\|_{L^\infty_tB_{2,  \infty}^0}$.
 Putting together (\ref{log34}),   (\ref{55}),  (\ref{thetaunique2})  and (\ref{thetaunique1}) gives
\begin{equation*}
X(t)\le f(t)\Big(X(0)+\int_0^t\|v^1(\tau)\|_{B_{\infty,  1}^1}\mu(X(\tau))d\tau\Big),  
\end{equation*}
with $f$ a known function depending continuously and increasingly on the quantities 
 $\|(v^i,  \theta^i)\|_{\mathcal{X}_t}$  and on the variable time. Now from Lemma \ref{osgood} we get the uniqueness. 

Finally,   let us now give some quantified estimates that will be used later for the 
construction of the solutions. Applying \mbox{Remark \ref{osg23}}  we get 
\begin{equation}
\label{uniqos}
X(0) \le \alpha(T)\Longrightarrow X(t)\le \beta(T)\big(X(0)\big)^{\gamma(T)},  
\end{equation}
 where $\alpha, \beta,\gamma$ are explicit functions  depending continuously  on $\|(v^i,  \theta^i)\|_{\mathcal{X}_T}$  and $T$.

\subsection{Existence}

We consider the following system
\begin{equation} 
\left\{ \begin{array}{ll} 
\partial_{t}v_n+v_n\cdot\nabla v_n+|\DD| v_n+\nabla p_n=\theta_n e_{2}\\ 
\partial_{t}\theta_n+v_n\cdot\nabla\theta_n=0\\
\textnormal{div}v_n=0\\
{v_n}_{| t=0}=S_nv^{0},   \quad {\theta_n}_{| t=0}=S_n\theta^{0}  
\end{array} \right. \tag{B$_{n}$}
\end{equation}
First remark that $S_n v^0,  S_n\theta^0\in H^s,   \forall s\in\RR$  since $v^0,  \theta^0 \in L^2.$ 
 As in the classical theory of  quasi-linear hyperbolic systems,   
  we can prove the local well-posedness of the system (${\rm B}_n$). The global well-posedness is related to the following criterion: the solution can be continued beyond the time $T$ if the quantity $\|\nabla v_n\|_{L^1_T L^\infty}$ is finite. Now from the {\it a priori} estimates,    in particular Proposition \ref{prons1},   the Lipschitz norm
   of the velocity  can not blow up in finite time and  hence  the solution $(v_n,  \theta_n)$ is globally defined. Once again from the {\it a priori} estimates   we have for $1\le\rho<{p}/{2}$ 
$$
\|v_n\|_{L^\infty_T (H^1\cap\dot{W}^{1,  p})}+\|v_n\|_{ L^\rho_T B_{\infty,  1}^1}\leq \Phi_2(T),
$$
and
$$
\|\theta_n\|_{L^\infty_T(L^2\cap B_{\infty,  1}^{0})}\leq \Phi_2(T).
$$
It follows that  up  to the extraction of a subsequence $(v_n,  \theta_n)$ is weakly convergent to $(v,  \theta)$ satisfying the same estimate as above.
Now using (\ref{uniqos})  we get the following:  if we have $$a_{n,  m}=\|(S_n-S_m) v^0\|_{B_{2,  \infty}^0}+\|(S_n-S_m )\theta^0\|_{B_{2,  \infty}^{-1}}\le \alpha(T)
 $$
  then
\begin{eqnarray*}
\|v_n-v_m\|_{L^\infty_TB_{2,  \infty}^0}+\|\theta_n-\theta_m\|_{L^\infty_T B_{2,  \infty}^{-1}}\leq \beta(T)\big(a_{n,  m}\big)^{\gamma(T)}.  
\end{eqnarray*}
This proves that $(v_n)$ is a Cauchy sequence and hence that  it converges strongly to $v$ in the space $L^\infty_TB_{2,  \infty}^0.$ 
By interpolation   we can easily get   the strong convergence of $v_{n}$ to $v$  in $ L^2([0,T]\times\RR^2)$. This  implies that
 $ v_{n} \otimes v_{n}$ converges  in $ L^1([0,T]\times\RR^2)$. But since
 $\theta_n$ converges to $\theta$ weakly in $ L^2([0,T]\times\RR^2)$ then, by weak strong convergence,  we have also that
   $ v_{n}\, \theta_{n}$ converges weakly  to $v\, \theta$.

This allows us to pass to the limit in the system $({\rm B}_n)$ and  to  get that $(v,  \theta)$ is a solution of our initial problem.
\section*{Appendix: some technical lemmas}
\label{technical}
Here we restate and prove Lemma \ref{propDelta}.
\begin{lemm}
\item Let    $\beta\in [2,  +\infty[$,   $s\in]0,  1[$ and $u\in L^{2\beta}\cap\dot{H}^{s+1-\frac2\beta}$. Then we have
\begin{equation*}
\||u|^{\beta-2}u\|_{\dot{H}^s}\lesssim\|u\|_{L^{2\beta}}^{\beta-2}\|u\|_{\dot{H}^{s+1-\frac2\beta}}.
\end{equation*}
\end{lemm}
\begin{proof}
We shall actually establish the  more accurate estimate:
\begin{equation*}
\||u|^{\beta-2}u\|_{\dot{H}^s}\le C\|u\|_{L^{2\beta}}^{\beta-2}\|u\|_{\dot{B}_{\beta,  2}^s}.
\end{equation*}
Once this estimate is established,   the result follows from   the embedding \mbox{$\dot{H}^{s+1-\frac2\beta}\hookrightarrow\dot{B}_{\beta,  2}^s,  $} for $\beta\geq2$ which is an easy consequence of Bernstein inequalities.
For $0<s<1$,    we can use the characterization of the homogeneous Sobolev space $\dot H^s,  $
\begin{equation}\label{chara}
\big\||u|^{\beta-2}u\big\|_{\dot{H}^s}^2\approx\int_{\RR^2}\frac{\||u|^{\beta-2}u(x-\cdot)-|u|^{\beta-2}u(\cdot)\|_{L^2}^2}{|x|^{2+2s}}dx.
\end{equation}
On the other hand there exists $C$ depending on $\beta$ such that  for every $a,  b\in\RR$
$$
\big||a|^{\beta-2} a-|b|^{\beta-2}b \big|\leq C|a-b|\big(|a|^{\beta-2}+|b|^{\beta-2}\big).
$$
Thus using this inequality and  integrating in $y$ we get by Cauchy-Schwarz inequality 
$$
\||u|^{\beta-2}u(x-\cdot)-|u|^{\beta-2}u(\cdot)\|_{L^2}\leq C\|u(x-\cdot)-u(\cdot)\|_{L^\beta}\|u\|_{L^{2\beta}}^{\beta-2}.
$$
Inserting this estimate into (\ref{chara}) and using the characterization of Besov space leads to
\begin{eqnarray*}
\||u|^{\beta-2}u\|_{\dot{H}^s}^2&\lesssim& \|u\|_{L^{2\beta}}^{2\beta-4}\int_{\RR^2}\frac{\|u(x-\cdot)-
u(\cdot)\|_{L^\beta}^2}{|x|^{2+2s}}dx\\
&\lesssim&\|u\|_{L^{2\beta}}^{2\beta-4}\|u\|_{\dot{B}_{\beta,  2}^s}^2.
\end{eqnarray*}
This concludes the proof.
\end{proof}

\begin{lemm}[Commutators estimates]
\label{lemproduit}
Let $v$ be a smooth divergence-free vector field and $f$ be a smooth function then
\begin{enumerate}
\item for every $q\geq -1$
$$
\|[\Delta_q,   v\cdot\nabla]f\|_{L^p}\lesssim \|\nabla v\|_{L^p}\|f\|_{B_{\infty,  \infty}^{0}}.
$$
\item For every  $s\in [-1,  0]$
$$
\|v\cdot\nabla f\|_{B_{2,  \infty}^{s}}\lesssim\|v \|_{L^2}\|f\|_{B_{\infty,  1}^{1+s}}.
$$
\end{enumerate}
\end{lemm}
\begin{proof}

(1)   We shall actually  prove the refined estimate
$$
\|[\Delta_q,   v\cdot\nabla]\theta\|_{L^p}\lesssim \|\nabla v\|_{L^p}\|\theta\|_{B_{\infty,  \infty}^{0}}.
$$
The desired estimate will follow from the embedding  \mbox{$B_{p,  1}^{\frac2p}\hookrightarrow B_{\infty,  \infty}^0.$} We have from Bony's decomposition
\begin{eqnarray*}
\label{dec0}
[\Delta_q,   v\cdot\nabla]\theta&=&\sum_{|j-q|\le 4}[\Delta_q,   S_{j-1}v\cdot\nabla]
\Delta_j\theta+\sum_{|j-q|\le4}[\Delta_q,   \Delta_jv\cdot\nabla]S_{j-1}\theta\\
\nonumber&+&\sum_{j\geq q-4} [\Delta_q,   
\Delta_jv\cdot\nabla]\widetilde{\Delta}_j\theta\\
&:=& \mbox{I}_q+\mbox{II}_q+\mbox{III}_q.
\end{eqnarray*}
Observe first that 
 $$
\mbox{I}_q=\sum_{|j-q|\le 4}h_q\star( S_{j-1}v\cdot\nabla\Delta_j\theta)-S_{j-1}v\cdot(h_q\star \nabla\Delta_j\theta)
$$
where  $\hat h_q(\xi)=\varphi({2^{-q}}\xi)$. Thus,   Lemma \ref{commu} and Bernstein inequalities yield
\begin{eqnarray*}
\nonumber\|\mbox{I}_q\|_{L^p}&\lesssim&\sum_{|j-q|\le4}\|xh_q\|_{L^1}\|\nabla S_{j-1} 
v\|_{L^p}\|\nabla\Delta_j\theta\|_{L^\infty}\\
\nonumber
&\lesssim&\|\nabla  v\|_{L^p}\|h_0\|_{L^1}\sum_{|j-q|\le4}2^{j-q}\|\Delta_j\theta\|_{L^\infty}\\
&\lesssim&\|\nabla v\|_{L^p}\|\theta\|_{B_{\infty,  \infty}^0}.
\end{eqnarray*}
To estimate  the second term we use once again Lemma \ref{commu}
\begin{eqnarray*}
\|\hbox{II}_q\|_{L^p}&\lesssim&\sum_{|j-q|\le 4}2^{-q}\|\Delta_j\nabla v\|_{L^p}\|\nabla S_{j-1}\theta\|_{L^\infty}\\&\lesssim&
 \|\nabla v\|_{L^p}\sum_{|j-q|\le 4\atop k\le j-2}2^{k-q}\|\Delta_k\theta\|_{L^\infty}\\
&\lesssim&\|\nabla v\|_{L^p}\|\theta\|_{B_{\infty,  \infty}^0}.
\end{eqnarray*}
 Let us now move to the remainder term. We separate it into two terms: high frequencies and low frequencies.
\begin{eqnarray*}
\hbox{III}_q&=&\sum_{j\geq q-4\atop j\in\NN} [\Delta_q\partial_i,   \Delta_jv^i]\widetilde{\Delta}_j\theta+ [\Delta_q,   \Delta_{-1}v\cdot\nabla]\widetilde{\Delta}_{-1}\theta\\
&:=&\hbox{III}_q^1+\hbox{III}_q^2.
\end{eqnarray*}
For the first term we don't need to use the structure of the commutator. We estimate separately each term of the commutator by using Bernstein inequalities.
\begin{eqnarray*}
\|\hbox{III}_q^1\|_{L^p}&\lesssim& \sum_{j\geq q-4\atop j\in\NN} 2^q\|\Delta_jv\|_{L^p}\|\widetilde\Delta_j\theta\|_{L^\infty}\\
&\lesssim&
 \|\nabla v\|_{L^p}\sum_{j\geq q- 4}2^{q-j}\|\widetilde\Delta_j\theta\|_{L^\infty}\\
&\lesssim&\|\nabla v\|_{L^p}\|\theta\|_{B_{\infty,  \infty}^0}.
\end{eqnarray*}
For the second term we use Lemma \ref{commu} combined with Bernstein inequalities.
\begin{eqnarray*}
\|\hbox{III}_q^2\|_{L^p}&\lesssim&\|\nabla\Delta_{-1} v\|_{L^p}\|\nabla\widetilde\Delta_{-1}\theta\|_{L^\infty}\\
&\lesssim&\|\nabla v\|_{L^p}\|\theta\|_{B_{\infty,  \infty}^0}.
\end{eqnarray*}

(2)
According to  by Bony's decompostion and  $\hbox{div} v=0$
$$
v\cdot\nabla u=T_{v^i}\partial_iu+T_{\partial_iu} v^i+\partial_i\mathcal{R}(v^i,  u).
$$
To estimate the first term we write by definition
$$
2^{qs}\|\Delta_q(T_{v^i}\partial_iu)\|_{L^2}\lesssim\sum_{|j-q|\le 4}2^{j+qs}\|S_{j-1}v\|_{L^2}\|\Delta_ju\|_{L^\infty}.
$$
Now we use the inequality
$$
\|S_{j-1}v\|_{L^2}\le \|v\|_{L^2}.
$$
Thus we get
$$
\sup_{q}2^{qs}\|\Delta_q(T_{v^i}\partial_iu)\|_{L^2}\lesssim \|v\|_{L^2}\|u\|_{B_{\infty,  \infty}^{1+s}}.
$$
Straightforward calculus gives since $1+s\le0,  $
\begin{eqnarray*}
2^{qs}\|\Delta_q(T_{\partial_iu}v^i)\|_{L^2}&\lesssim&\sum_{|j-q|\le4\atop  k\le j-1}2^{j(1+s)}\|\Delta_k u\|_{L^\infty}\|\Delta_jv\|_{L^2}\\
&\lesssim&\|v\|_{L^2}\|u\|_{B_{\infty,  1}^{1+s}}.
\end{eqnarray*}
and
$$
\|\partial_i\mathcal{R}(v^i,  u)\|_{B_{2,  \infty}^{s}}\lesssim \|v\|_{L^2}\|u\|_{B_{\infty,  1}^{1+s}}.
$$

\end{proof}

\begin{lemm}\label{log-inter}
Let $v\in H^1$ then we have
$$
\|v\|_{L^2}\lesssim \|v\|_{B_{2,  \infty}^0}\log\Big(e+\frac{\|v\|_{H^1}}{\|v\|_{B_{2,  \infty}^0}}    \Big).
$$
\end{lemm}
\begin{proof}
Let $N\in \NN^*$ be a fixed number then using the dyadic decomposition we get
\begin{eqnarray*}
\|v\|_{L^2}&\le&\sum_{q\le N-1}\|\Delta_q v\|_{L^2}+\sum_{q\geq N}\|\Delta_q v\|_{L^2}\\
&\lesssim&N\|v\|_{B_{2,  \infty}^0}+2^{-N}\|v\|_{H^1}.
\end{eqnarray*}
Choosing $$
N=\Big[\log_2\Big(e+\|v\|_{H^1}/\|v\|_{B_{2,  \infty}^0}\Big)\Big]
$$
gives the desired result.
\end{proof}

\end{document}